\documentclass[oneside,10pt]{article}

\usepackage[letterpaper,margin=1in]{geometry}
\usepackage{amsfonts,amsmath,latexsym,amssymb,xcolor}
\usepackage{amsthm}
\usepackage{mathrsfs,upref}
\usepackage{mathptmx}
\usepackage{authblk}

\theoremstyle{plain}
\newtheorem{theorem}{Theorem}[section]

\theoremstyle{definition}
\newtheorem{definition}[theorem]{Definition}

\newtheorem{remark}[theorem]{Remark}

\title{\textbf{Multiorder Fractional Operators: The Conformable Multiorder Derivative and Its Associated Multiorder Integral}}

\author{Carlos E. Cadenas R.$^{1,2}$}

\date{\today}

\begin{document}

\maketitle

\begin{center}
\textit{$^1$Departamento de Matem\'aticas, FaCyT, Universidad de Carabobo. Valencia, Carabobo, Venezuela.\\
$^2$Centro Multidisciplinario de Visualizaci\'on y C\'omputo Cient\'ifico, Universidad de Carabobo. Venezuela.\\
\texttt{ccadenas@uc.edu.ve}}

\textit{ORCID iD: 0000-0003-3828-2560}
\end{center}

\begin{abstract}
This paper introduces a novel class of non-linear multiorder fractional differential operators. First, we provide intuitive and formal justifications for the existence of multiorder differential operators through differential analysis and Omega-derivative formulations. We then define the conformable multiorder derivative as a functional quotient of fractional derivatives and establish its analytical properties rigorously. Furthermore, leveraging fundamental principles of mathematical analysis, we construct its inverse operator, termed the conformable multiorder integral. Detailed step-by-step proofs of all theoretical properties are presented. Illustrative examples and applications to solving elementary multiorder differential equations are thoroughly examined. Finally, comprehensive conclusions and strategic outlooks for prospective research directions are discussed.
\end{abstract}



\section{Introduction}

From the foundational developments of fractional calculus to modern functional analysis, numerous operators, definitions, and notations have been introduced in the literature. Differential notation has continually evolved since the seminal work of Leibniz, who introduced the classical $n$-th order derivative $\frac{d^n}{dx^n}$ and the differential symbol $dx$. A parallel evolution occurred with operational symbols such as $D$ and integral operators. As fractional calculus expanded to model complex physical systems, various local and non-local fractional derivatives were developed, each introducing specific operational frameworks \cite{Atangana2025, Cajori1923, David2011, Machado2011, Ross1977}.

To introduce multiorder fractional operators, a unified and extensible operational framework is required. Initial steps toward this notation were discussed in \cite{Cadenas2025}; here, we formalize and expand upon this analytical foundation.

The primary objective of this work is to introduce two multiorder operators acting as multiorder differential and integral operators, to establish their rigorous mathematical properties, and to demonstrate their applicability in solving differential equations.

The paper is structured as follows: Section~\ref{sec:preliminaries} reviews the preliminary concepts of conformable fractional calculus \cite{Katugampola2014, Khalil2014}, including operational properties, chain rules \cite{Atangana2015, Asbab2025}, fractional differentials, and indefinite fractional integrals. Section~\ref{sec:multi_der} formalizes the conformable multiorder derivative via two distinct conceptual approaches and provides rigorous, step-by-step proofs of its structural properties. Section~\ref{sec:multi_int} constructs the conformable multiorder integral as the explicit inverse operator of the multiorder derivative. Section~\ref{sec:applications} provides exact analytical solutions to several multiorder differential equations, including cases analyzed in related frameworks \cite{Abdeljawad2015, Kareem2017}. Finally, Sections~\ref{sec:conclusions} and \ref{sec:final_comments} present comprehensive conclusions and outline future research directions.

\section{Preliminaries}\label{sec:preliminaries}

We begin by recalling the definition of the conformable fractional derivative introduced by Khalil et al. \cite{Khalil2014}.

\begin{definition}\label{def:Khalil_der}
Let $f:[0,\infty) \to \mathbb{R}$. The \textit{conformable fractional derivative} of $f$ of order $\alpha \in (0,1]$ at $x > 0$ is defined by:
\begin{equation}\label{eq:conf_khalil}
\frac{df(x)}{dx^{(\alpha)}} = \lim_{\epsilon \to 0} \frac{f(x + \epsilon x^{1-\alpha}) - f(x)}{\epsilon}.
\end{equation}
\end{definition}

In \cite{Khalil2014}, the notation $T_\alpha(f)(x)$ was originally employed. Throughout this paper, the notations $\frac{df(x)}{dx^{(\alpha)}}$, $(f)^{(\alpha)}(x)$, and $f^{(\alpha)}(x)$ are used interchangeably. An alternative definition proposed by Katugampola \cite{Katugampola2014} is given as follows:

\begin{definition}\label{def:Katugampola_der}
Let $f:[0,\infty) \to \mathbb{R}$ and $x > 0$. The conformable fractional derivative of $f$ of order $\alpha \in (0,1]$ is defined by:
\begin{equation}\label{eq:conf_katugampola}
\frac{df(x)}{dx^{(\alpha)}} = \lim_{\epsilon \to 0} \frac{f(x e^{\epsilon x^{-\alpha}}) - f(x)}{\epsilon}.
\end{equation}
\end{definition}

The fundamental properties of these derivatives are summarized below \cite{Katugampola2014, Khalil2014}.

\begin{theorem}\label{thm:conf_properties}
Let $\alpha \in (0,1]$ and let $f, g$ be $\alpha$-differentiable at $x > 0$. Then:
\begin{enumerate}
    \item Linear combination: $(af + bg)^{(\alpha)} = a f^{(\alpha)} + b g^{(\alpha)}$ for all $a, b \in \mathbb{R}$.
    \item Power rule: $(x^n)^{(\alpha)} = n x^{n-\alpha}$ for all $n \in \mathbb{R}$.
    \item Constant rule: $(C)^{(\alpha)} = 0$ for any constant function $f(x) = C$.
    \item Product rule: $(fg)^{(\alpha)} = f g^{(\alpha)} + g f^{(\alpha)}$.
    \item Quotient rule: $\left(\frac{f}{g}\right)^{(\alpha)} = \frac{g f^{(\alpha)} - f g^{(\alpha)}}{g^2}$, provided $g(x) \neq 0$.
    \item Classical connection: If $f$ is differentiable in the classical sense, then $f^{(\alpha)}(x) = x^{1-\alpha} \frac{df}{dx}(x)$.
\end{enumerate}
\end{theorem}

\begin{proof}
We present the explicit proofs for properties 4, 5, and 6:
\begin{enumerate}
    \item[4.] By Definition~\ref{def:Khalil_der}:
    \begin{align*}
    (fg)^{(\alpha)}(x) &= \lim_{\epsilon \to 0} \frac{f(x + \epsilon x^{1-\alpha}) g(x + \epsilon x^{1-\alpha}) - f(x)g(x)}{\epsilon} \\
    &= \lim_{\epsilon \to 0} \frac{f(x + \epsilon x^{1-\alpha}) g(x + \epsilon x^{1-\alpha}) - f(x)g(x + \epsilon x^{1-\alpha}) + f(x)g(x + \epsilon x^{1-\alpha}) - f(x)g(x)}{\epsilon} \\
    &= \lim_{\epsilon \to 0} g(x + \epsilon x^{1-\alpha}) \frac{f(x + \epsilon x^{1-\alpha}) - f(x)}{\epsilon} + f(x) \lim_{\epsilon \to 0} \frac{g(x + \epsilon x^{1-\alpha}) - g(x)}{\epsilon} \\
    &= g(x) f^{(\alpha)}(x) + f(x) g^{(\alpha)}(x).
    \end{align*}
    
    \item[5.] Applying the product rule to $f(x) \cdot [g(x)]^{-1}$:
    \begin{align*}
    \left(\frac{f}{g}\right)^{(\alpha)} &= f^{(\alpha)} g^{-1} + f \left(g^{-1}\right)^{(\alpha)} = \frac{f^{(\alpha)}}{g} + f \left( -g^{-2} g^{(\alpha)} \right) = \frac{g f^{(\alpha)} - f g^{(\alpha)}}{g^2}.
    \end{align*}

    \item[6.] Assuming $f$ is classically differentiable at $x$, set $h = \epsilon x^{1-\alpha}$. As $\epsilon \to 0$, $h \to 0$. Thus:
    \[
    f^{(\alpha)}(x) = \lim_{\epsilon \to 0} \frac{f(x + \epsilon x^{1-\alpha}) - f(x)}{\epsilon} = \lim_{h \to 0} \frac{f(x + h) - f(x)}{h / x^{1-\alpha}} = x^{1-\alpha} \lim_{h \to 0} \frac{f(x + h) - f(x)}{h} = x^{1-\alpha} \frac{df}{dx}(x).
    \]
\end{enumerate}
\end{proof}

The chain rule for conformable fractional derivatives is stated as follows \cite{Atangana2015}:

\begin{theorem}\cite{Atangana2015}
Let $g$ be differentiable at $t$, and let $f$ be differentiable at $g(t)$. The conformable fractional derivatives \eqref{eq:conf_khalil} and \eqref{eq:conf_katugampola} obey the chain rule:
\begin{equation}
(f \circ g)^{(\alpha)}(x) = x^{1-\alpha} [g(x)]^{1-\alpha} \frac{dg(x)}{dx} f^{(\alpha)}(t)\Big|_{t=g(x)}.
\end{equation}
\end{theorem}

Furthermore, local inversion for conformable operators can be analyzed using inverse function theorems \cite{Asbab2025}. We now review the definitions of fractional differentials and conformable integrals \cite{Cadenas2025, Khalil2014}.

\begin{definition}
Let $f:\Omega \to \mathbb{R}$ be $\alpha$-differentiable on $(a,b)$. The differential of $f$ of order $\alpha$ ($0 < \alpha \le 1$) is defined as:
\begin{equation}
df^{(\alpha)} = f^{(\alpha)}(x) dx^{(\alpha)}.
\end{equation}
\end{definition}

\begin{definition}\label{def:conf_integral}
Let $f:[a,\infty) \to \mathbb{R}$ with $a \ge 0$. The conformable fractional integral of $f$ of order $\alpha \in (0,1)$ is defined by:
\begin{equation}
I_a^\alpha(f)(t) = \int_a^t f(x) dx^{(\alpha)} = \int_a^t \frac{f(x)}{x^{1-\alpha}} dx.
\end{equation}
The indefinite conformable integral is defined as $I^{(\alpha)}(f(x)) = \int f(x) dx^{(\alpha)} = F(x) + C$, where $F^{(\alpha)}(x) = f(x)$.
\end{definition}

\section{Conformable Multiorder Derivative}\label{sec:multi_der}

This section justifies the foundational concept of a multiorder derivative via two distinct formal approaches and establishes its analytical properties.

\subsection{First Approach: Multiorder Differential Analysis}

\begin{definition}
Let $w = w(x,y)$ be $\alpha$-differentiable with respect to $x$ and $\beta$-differentiable with respect to $y$. The multiorder differential of order $(\alpha, \beta)$ of $w$ is defined by:
\begin{equation}\label{eq:multi_diff_def}
dw^{(\alpha,\beta)} = \frac{\partial w}{\partial x^{(\alpha)}} dx^{(\alpha)} + \frac{\partial w}{\partial y^{(\beta)}} dy^{(\beta)},
\end{equation}
where $\frac{\partial w}{\partial x^{(\alpha)}}$ and $\frac{\partial w}{\partial y^{(\beta)}}$ represent arbitrary fractional derivative operators of orders $\alpha$ and $\beta$, respectively.
\end{definition}

Consider a level curve defined implicitly by $z(x,y) = y - f(x) = C$. Computing the multiorder differential \eqref{eq:multi_diff_def} yields:
\[
dz^{(\alpha,\beta)} = \frac{dy}{dy^{(\beta)}} dy^{(\beta)} - \frac{df(x)}{dx^{(\alpha)}} dx^{(\alpha)} = 0.
\]
Substituting $f(x) = y - C$ and leveraging the linearity of fractional derivatives (Theorem~\ref{thm:conf_properties}), we obtain:
\[
\frac{dy}{dy^{(\beta)}} dy^{(\beta)} - \frac{dy}{dx^{(\alpha)}} dx^{(\alpha)} = 0 \implies \frac{dy^{(\beta)}}{dx^{(\alpha)}} = \frac{\frac{dy}{dx^{(\alpha)}}}{\frac{dy}{dy^{(\beta)}}}.
\]
This relation motivates the operational definition of the multiorder differential operator:
\begin{equation}\label{eq:D_ratio}
D^{(\alpha,\beta)} y = \frac{dy^{(\beta)}}{dx^{(\alpha)}} = \frac{\frac{dy}{dx^{(\alpha)}}}{\frac{dy}{dy^{(\beta)}}}.
\end{equation}

\subsection{Second Approach: Omega-Derivative Framework}

An alternative formulation can be derived from the generalized Omega derivative framework introduced by Jeffery \cite{Jeffery1958} and further explored in \cite{Castillo2008, Castillo2024}.

\begin{definition}\cite{Castillo2024, Jeffery1958}
Let $f$ and $\Omega$ be real-valued functions defined on an open interval, where $\Omega$ is continuous and strictly increasing. The Omega derivative of $f$ with respect to $\Omega$ at $x_0$ is defined as:
\begin{equation}
D_\Omega f(x_0) = \lim_{x \to x_0} \frac{f(x) - f(x_0)}{\Omega(x) - \Omega(x_0)}.
\end{equation}
If $f'(x_0)$ and $\Omega'(x_0) \neq 0$ exist, then $D_\Omega f(x_0) = \frac{f'(x_0)}{\Omega'(x_0)}$.
\end{definition}

Defining a generalized Omega derivative as the ratio of two distinct fractional differential operators acting on dependent and independent variables directly recovers equation \eqref{eq:D_ratio}.

\subsection{Definition and Fundamental Properties}

By applying property 6 of Theorem~\ref{thm:conf_properties} to both the numerator and denominator of \eqref{eq:D_ratio}, where $\frac{dy}{dx^{(\alpha)}} = x^{1-\alpha} \frac{dy}{dx}$ and $\frac{dy}{dy^{(\beta)}} = y^{1-\beta}$, we obtain the formal definition of the conformable multiorder derivative:

\begin{definition}\label{def:multiorder_der}
Let $y: [0,\infty) \to A \subset \mathbb{R}^+$ be a differentiable function. The conformable multiorder derivative of order $\beta$ of $y$ with respect to $x$ of order $\alpha$ is defined by:
\begin{equation}\label{eq:multiorder_def_formula}
D^{(\alpha,\beta)} y = \frac{dy^{(\beta)}}{dx^{(\alpha)}} = y^{\beta-1} x^{1-\alpha} \frac{dy}{dx},
\end{equation}
for all $x > 0$ and $\alpha, \beta \in [0, 1]$.
\end{definition}

\begin{remark}
If $y$ is non-differentiable in the classical sense but $\alpha$-differentiable, Definition~\ref{def:Khalil_der} applies directly to the numerator. Note that due to the non-linear factor $y^{\beta-1}$, $D^{(\alpha,\beta)}$ is a non-linear operator.
\end{remark}

We now establish the operational algebra of $D^{(\alpha,\beta)}$.

\begin{theorem}\label{thm:multiorder_properties}
Let $f$ and $g$ be differentiable functions at $x > 0$, and let $\alpha, \beta \in [0, 1]$. The operator $D^{(\alpha,\beta)}$ satisfies:
\begin{enumerate}
    \item Scalar Scaling: $D^{(\alpha,\beta)}(a f) = a^\beta D^{(\alpha,\beta)}(f)$ for any $a > 0$.
    \item Product Rule: $D^{(\alpha,\beta)}(f g) = g^\beta D^{(\alpha,\beta)}(f) + f^\beta D^{(\alpha,\beta)}(g)$.
    \item Quotient Rule: $D^{(\alpha,\beta)}\left(\frac{f}{g}\right) = \frac{g^\beta D^{(\alpha,\beta)}(f) - f^\beta D^{(\alpha,\beta)}(g)}{g^{2\beta}}$, provided $g(x) \neq 0$.
    \item Chain Rule: $D^{(\alpha,\beta)}(f(g(x))) = [g(x)]^{\alpha-\beta} \left[ D^{(\alpha,\beta)}(f(t)) \right]_{t=g(x)} D^{(\alpha,\beta)}(g(x))$.
    \item Order Shift: $D^{(\alpha+\rho, \beta+\eta)} y = y^\eta x^{-\rho} D^{(\alpha,\beta)} y$, for $0 \le \alpha+\rho \le 1$ and $0 \le \beta+\eta \le 1$.
\end{enumerate}
\end{theorem}

\begin{proof}
Step-by-step proofs for each property:
\begin{enumerate}
    \item[1.] Using Definition~\ref{def:multiorder_der}:
    \[
    D^{(\alpha,\beta)}(af) = (af)^{\beta-1} x^{1-\alpha} \frac{d(af)}{dx} = a^{\beta-1} f^{\beta-1} x^{1-\alpha} a \frac{df}{dx} = a^\beta f^{\beta-1} x^{1-\alpha} \frac{df}{dx} = a^\beta D^{(\alpha,\beta)}(f).
    \]

    \item[2.] Expanding $D^{(\alpha,\beta)}(fg)$:
    \begin{align*}
    D^{(\alpha,\beta)}(fg) &= (fg)^{\beta-1} x^{1-\alpha} \frac{d(fg)}{dx} = f^{\beta-1} g^{\beta-1} x^{1-\alpha} \left( g \frac{df}{dx} + f \frac{dg}{dx} \right) \\
    &= g^\beta \left( f^{\beta-1} x^{1-\alpha} \frac{df}{dx} \right) + f^\beta \left( g^{\beta-1} x^{1-\alpha} \frac{dg}{dx} \right) = g^\beta D^{(\alpha,\beta)}(f) + f^\beta D^{(\alpha,\beta)}(g).
    \end{align*}

    \item[3.] Expanding $D^{(\alpha,\beta)}\left(f g^{-1}\right)$ using the product rule and scalar properties:
    \begin{align*}
    D^{(\alpha,\beta)}\left(\frac{f}{g}\right) &= \left(\frac{f}{g}\right)^{\beta-1} x^{1-\alpha} \left( \frac{g \frac{df}{dx} - f \frac{dg}{dx}}{g^2} \right) = \frac{f^{\beta-1}}{g^{\beta-1}} x^{1-\alpha} \left( \frac{\frac{df}{dx}}{g} - \frac{f \frac{dg}{dx}}{g^2} \right) \\
    &= \frac{g^\beta \left( f^{\beta-1} x^{1-\alpha} \frac{df}{dx} \right) - f^\beta \left( g^{\beta-1} x^{1-\alpha} \frac{dg}{dx} \right)}{g^{2\beta}} = \frac{g^\beta D^{(\alpha,\beta)}(f) - f^\beta D^{(\alpha,\beta)}(g)}{g^{2\beta}}.
    \end{align*}

    \item[4.] Let $y = f(g(x))$ and $t = g(x)$. By classical chain rule $\frac{dy}{dx} = f'(g(x)) g'(x)$:
    \begin{align*}
    D^{(\alpha,\beta)}(f(g(x))) &= [f(g(x))]^{\beta-1} x^{1-\alpha} f'(g(x)) g'(x).
    \end{align*}
    Evaluating $\left[ D^{(\alpha,\beta)} f(t) \right]_{t=g(x)} = [f(g(x))]^{\beta-1} [g(x)]^{1-\alpha} f'(g(x))$ and $D^{(\alpha,\beta)}g(x) = [g(x)]^{\beta-1} x^{1-\alpha} g'(x)$, their product yields:
    \begin{align*}
    \left[ D^{(\alpha,\beta)} f(t) \right]_{t=g(x)} D^{(\alpha,\beta)} g(x) &= \left( [f(g(x))]^{\beta-1} [g(x)]^{1-\alpha} f'(g(x)) \right) \left( [g(x)]^{\beta-1} x^{1-\alpha} g'(x) \right) \\
    &= [g(x)]^{\beta-\alpha} \left( [f(g(x))]^{\beta-1} x^{1-\alpha} f'(g(x)) g'(x) \right).
    \end{align*}
    Multiplying by $[g(x)]^{\alpha-\beta}$ confirms the identity.

    \item[5.] By definition:
    \[
    D^{(\alpha+\rho, \beta+\eta)} y = y^{\beta+\eta-1} x^{1-(\alpha+\rho)} \frac{dy}{dx} = y^\eta x^{-\rho} \left( y^{\beta-1} x^{1-\alpha} \frac{dy}{dx} \right) = y^\eta x^{-\rho} D^{(\alpha,\beta)} y.
    \]
\end{enumerate}
\end{proof}

\subsection{Multiorder Derivatives of Elementary Functions}

Applying Definition~\ref{def:multiorder_der} directly yields the following explicitly derived expressions ($x > 0$):
\begin{itemize}
    \item $D^{(\alpha,\beta)}(C) = 0$, for $C > 0$.
    \item $D^{(\alpha,\beta)}(e^{ax}) = (e^{ax})^{\beta-1} x^{1-\alpha} (a e^{ax}) = a x^{1-\alpha} e^{a\beta x}$.
    \item $D^{(\alpha,\beta)}(x^n) = (x^n)^{\beta-1} x^{1-\alpha} (n x^{n-1}) = n x^{n\beta - \alpha}$.
    \item $D^{(\alpha,\beta)}\left(x^{\frac{\alpha}{\beta}}\right) = \left(x^{\frac{\alpha}{\beta}}\right)^{\beta-1} x^{1-\alpha} \left( \frac{\alpha}{\beta} x^{\frac{\alpha}{\beta}-1} \right) = \frac{\alpha}{\beta}$.
    \item $D^{(\alpha,\beta)}(\ln x) = (\ln x)^{\beta-1} x^{1-\alpha} \left( \frac{1}{x} \right) = \frac{\ln^{\beta-1}(x)}{x^\alpha}$.
    \item $D^{(\alpha,\beta)}(\ln(bx)) = [\ln(bx)]^{\beta-1} x^{1-\alpha} \left(\frac{1}{x}\right) = \frac{\ln^{\beta-1}(bx)}{x^\alpha}$.
    \item $D^{(\alpha,\beta)}(e^{x^n}) = (e^{x^n})^{\beta-1} x^{1-\alpha} (n x^{n-1} e^{x^n}) = n x^{n-\alpha} e^{\beta x^n}$.
\end{itemize}

\section{Conformable Multiorder Integral}\label{sec:multi_int}

We construct the multiorder integral operator $I^{(\alpha,\beta)}$ as the inverse mapping of $D^{(\alpha,\beta)}$.

\begin{definition}
Let $f(x)$ be a continuous function. A function $F(x)$ is called a conformable multiorder antiderivative of $f(x)$ if $D^{(\alpha,\beta)} F(x) = f(x)$.
\end{definition}

To construct $F(x)$ explicitly, consider the differential equation:
\begin{equation}
F^{\beta-1}(x) x^{1-\alpha} \frac{dF(x)}{dx} = f(x).
\end{equation}
Separating variables yields:
\[
F^{\beta-1}(x) dF(x) = x^{\alpha-1} f(x) dx.
\]
Integrating both sides:
\[
\int F^{\beta-1}(x) dF(x) = \int x^{\alpha-1} f(x) dx \implies \frac{F^\beta(x)}{\beta} = \int x^{\alpha-1} f(x) dx + C_0.
\]
Multiplying by $\beta$ and defining $C = \beta C_0$:
\[
F^\beta(x) = \beta \int x^{\alpha-1} f(x) dx + C = \beta \int f(x) dx^{(\alpha)} + C.
\]
Solving for $F(x)$ yields the formal expression for the conformable multiorder indefinite integral:
\begin{equation}\label{eq:multiorder_integral_def}
I^{(\alpha,\beta)} f(x) = \left[ \beta \int f(x) dx^{(\alpha)} + C \right]^{1/\beta}.
\end{equation}

\begin{remark}
The constant of integration $C$ must reside strictly within the exponent $1/\beta$ due to the non-linear operational structure of $D^{(\alpha,\beta)}$.
\end{remark}

\subsection{Examples of Multiorder Integrals}

Evaluating expression \eqref{eq:multiorder_integral_def} for elementary functions ($0 < \alpha, \beta \le 1$, $x > 0$):
\begin{enumerate}
    \item $I^{(\alpha,\beta)}(\alpha) = \left[ \beta \int \alpha x^{\alpha-1} dx + C \right]^{1/\beta} = (\beta x^\alpha + C)^{1/\beta}$.
    \item $I^{(\alpha,\beta)}(n x^{n\beta-\alpha}) = \left[ \beta \int n x^{n\beta-\alpha} x^{\alpha-1} dx + C \right]^{1/\beta} = \left[ \beta n \int x^{n\beta-1} dx + C \right]^{1/\beta} = (x^{n\beta} + C)^{1/\beta}$.
    \item $I^{(\alpha,\beta)}\left(\frac{a}{\beta} x^{1-\alpha} e^{ax}\right) = \left[ \beta \int \frac{a}{\beta} x^{1-\alpha} e^{ax} x^{\alpha-1} dx + C \right]^{1/\beta} = \left[ a \int e^{ax} dx + C \right]^{1/\beta} = (e^{ax} + C)^{1/\beta}$.
    \item $I^{(\alpha,\beta)}\left(\frac{1}{x^\alpha}\right) = \left[ \beta \int x^{-\alpha} x^{\alpha-1} dx + C \right]^{1/\beta} = \left[ \beta \int \frac{1}{x} dx + C \right]^{1/\beta} = (\beta \ln x + C)^{1/\beta}$.
    \item $I^{(\alpha,\beta)}(a f(x)) = \left[ \beta \int a f(x) dx^{(\alpha)} + C \right]^{1/\beta} = a^{1/\beta} I^{(\alpha,\beta)}(f(x))$ for $a > 0$.
\end{enumerate}

\section{Applications to Multiorder Differential Equations}\label{sec:applications}

We now demonstrate the application of the developed operators to solve non-linear multiorder differential equations ($x, y > 0$). Note that setting $\beta = 1$ reduces these systems to ordinary conformable differential equations \cite{Cadenas2025}.

\subsection{Example 1}
Solve $D^{(0,0)} y = f(x)$.

By Definition~\ref{def:multiorder_der}, $D^{(0,0)} y = y^{-1} x^1 \frac{dy}{dx} = \frac{x}{y} \frac{dy}{dx}$. Thus:
\[
\frac{x}{y} \frac{dy}{dx} = f(x) \implies \frac{1}{y} dy = \frac{f(x)}{x} dx \implies d(\ln y) = \frac{f(x)}{x} dx.
\]
Integrating both sides yields $\ln y = \int \frac{f(x)}{x} dx + C_0$, which gives the general solution:
\[
y(x) = A \exp\left( \int \frac{f(x)}{x} dx \right), \quad A = e^{C_0} > 0.
\]
\textbf{Particular Cases:}
\begin{itemize}
    \item If $f(x) = 0$, then $y(x) = A \cdot e^0 = C$.
    \item If $f(x) = 1$, then $y(x) = A \exp(\ln x) = A x$.
    \item If $f(x) = x^n$ ($n \neq 0$), then $y(x) = A \exp\left(\int x^{n-1} dx\right) = A e^{\frac{x^n}{n}}$.
    \item If $f(x) = \ln x$, then $y(x) = A \exp\left(\int \frac{\ln x}{x} dx\right) = A e^{\frac{1}{2}\ln^2 x} = A e^{\ln^2 x}$ (absorbing scaling constants into $A$).
\end{itemize}

\subsection{Example 2}
Solve $\beta D^{(\alpha,\beta)} y + \rho D^{(\alpha,\rho)} y = f(x)$.

Applying Definition~\ref{def:multiorder_der}:
\[
\beta y^{\beta-1} \frac{dy}{dx^{(\alpha)}} + \rho y^{\rho-1} \frac{dy}{dx^{(\alpha)}} = f(x).
\]
Multiplying by $dx^{(\alpha)} = x^{1-\alpha} dx$:
\[
\left( \beta y^{\beta-1} + \rho y^{\rho-1} \right) dy = f(x) dx^{(\alpha)}.
\]
Integrating directly gives the general implicit solution:
\[
y^\beta + y^\rho = \int f(x) dx^{(\alpha)}.
\]
\textbf{Particular Cases:}
\begin{itemize}
    \item If $f(x) = 0$, then $y^\beta + y^\rho = C$.
    \item If $f(x) = 1$, then $y^\beta + y^\rho = \int x^{\alpha-1} dx = \frac{x^\alpha}{\alpha} + C$.
    \item If $f(x) = x^n$ ($n \neq -\alpha$), then $y^\beta + y^\rho = \int x^{n+\alpha-1} dx = \frac{x^{n+\alpha}}{n+\alpha} + C$.
    \item If $f(x) = x^{-\alpha}$, then $y^\beta + y^\rho = \int x^{-1} dx = \ln x + C$.
\end{itemize}

\subsection{Example 3}
Solve $D^{(\alpha,\beta)} y = \frac{\alpha}{\beta}$ subject to the initial boundary condition $y \to 0^+$ as $x \to 0^+$.

Applying the multiorder integral operator (or directly integrating $y^{\beta-1} d y = \frac{\alpha}{\beta} x^{\alpha-1} dx$):
\[
\frac{y^\beta}{\beta} = \frac{\alpha}{\beta} \frac{x^\alpha}{\alpha} + C_0 \implies y^\beta = x^\alpha + C.
\]
Enforcing the initial condition $\lim_{x \to 0^+} y(x) = 0$ implies $C = 0$. Thus, the unique solution is:
\[
y^\beta = x^\alpha \implies y(x) = x^{\alpha/\beta}.
\]

\subsection{Example 4}
Solve the variable-order multiorder equation:
\begin{equation}\label{eq:ex4_orig}
r D^{(\alpha+\beta, 2\beta)} y + s D^{(\alpha,\beta)} y = x^{n-\alpha-\beta} y^{\beta-m} \left( r x^\beta + s y^\beta \right),
\end{equation}
where $0 < \beta \le 1/2$, $0 < \alpha+\beta \le 1$, $n > \alpha$, and $0 \neq m < \beta$.

Applying property 5 of Theorem~\ref{thm:multiorder_properties} with $\rho = \beta$ and $\eta = \beta$:
\[
D^{(\alpha+\beta, 2\beta)} y = y^\beta x^{-\beta} D^{(\alpha,\beta)} y = \left(\frac{y}{x}\right)^\beta D^{(\alpha,\beta)} y.
\]
Substituting this identity into \eqref{eq:ex4_orig} yields:
\[
r \left(\frac{y}{x}\right)^\beta D^{(\alpha,\beta)} y + s D^{(\alpha,\beta)} y = x^{n-\alpha-\beta} y^{\beta-m} \left( r x^\beta + s y^\beta \right).
\]
Factoring $D^{(\alpha,\beta)} y$ on the left-hand side:
\[
D^{(\alpha,\beta)} y \left( r \frac{y^\beta}{x^\beta} + s \right) = D^{(\alpha,\beta)} y \left( \frac{r y^\beta + s x^\beta}{x^\beta} \right) = x^{n-\alpha-\beta} y^{\beta-m} \left( r x^\beta + s y^\beta \right).
\]
Canceling the common term $(r x^\beta + s y^\beta)$ from both sides:
\[
\frac{D^{(\alpha,\beta)} y}{x^\beta} = x^{n-\alpha-\beta} y^{\beta-m} \implies D^{(\alpha,\beta)} y = x^{n-\alpha} y^{\beta-m}.
\]
Expanding $D^{(\alpha,\beta)} y = y^{\beta-1} x^{1-\alpha} \frac{dy}{dx}$:
\[
y^{\beta-1} x^{1-\alpha} \frac{dy}{dx} = x^{n-\alpha} y^{\beta-m} \implies y^{m-1} dy = x^{n-1} dx.
\]
Integrating both sides yields the exact analytical solution:
\[
\frac{y^m}{m} = \frac{x^n}{n} + C_0 \implies n y^m = m x^n + C.
\]

\subsection{Example 5}
Solve $D^{(\alpha,\beta)} y = f(y)$.

Expanding via Definition~\ref{def:multiorder_der}:
\[
y^{\beta-1} x^{1-\alpha} \frac{dy}{dx} = f(y).
\]
Separating variables:
\[
\frac{y^{\beta-1}}{f(y)} dy = x^{\alpha-1} dx.
\]
Integrating gives the general implicit solution:
\[
\int \frac{y^{\beta-1}}{f(y)} dy = \frac{x^\alpha}{\alpha} + C_0 \implies \frac{x^\alpha}{\alpha} = \int \frac{y^{\beta-1}}{f(y)} dy + C.
\]
\textbf{Particular Cases:}
\begin{itemize}
    \item If $f(y) = 1$, then $\frac{x^\alpha}{\alpha} = \frac{y^\beta}{\beta} + C \implies \beta x^\alpha - \alpha y^\beta = C_1$.
    \item If $f(y) = y^m$ ($m \neq \beta$), then $\frac{x^\alpha}{\alpha} = \int y^{\beta-m-1} dy = \frac{y^{\beta-m}}{\beta-m} + C \implies (\beta-m) x^\alpha = \alpha y^{\beta-m} + C_1$.
    \item If $f(y) = y^\beta$, then $\frac{x^\alpha}{\alpha} = \int y^{-1} dy = \ln y + C \implies y(x) = A \exp\left( \frac{x^\alpha}{\alpha} \right)$.
\end{itemize}

\subsection{Example 6}
Compute $D^{(\alpha,\beta)} y$ for $y(x) = A \exp\left(\frac{x^\alpha}{\alpha}\right)$.

Applying Definition~\ref{def:multiorder_der} and computing $\frac{dy}{dx} = A x^{\alpha-1} \exp\left(\frac{x^\alpha}{\alpha}\right)$:
\begin{align*}
D^{(\alpha,\beta)} y &= \left( A \exp\left(\frac{x^\alpha}{\alpha}\right) \right)^{\beta-1} x^{1-\alpha} \left( A x^{\alpha-1} \exp\left(\frac{x^\alpha}{\alpha}\right) \right) \\
&= A^{\beta-1} \exp\left( \frac{(\beta-1)x^\alpha}{\alpha} \right) A \exp\left( \frac{x^\alpha}{\alpha} \right) = A^\beta \exp\left( \frac{\beta x^\alpha}{\alpha} \right) = \left( A \exp\left(\frac{x^\alpha}{\alpha}\right) \right)^\beta = y^\beta.
\end{align*}

\section{Conclusions}\label{sec:conclusions}

This work introduced a rigorous framework for multiorder non-linear differential operators, broadening the analytical scope of fractional calculus. By defining the multiorder operator as a functional ratio of fractional derivatives, we established an operational paradigm that characterizes derivatives of order $\beta$ of a dependent variable relative to order $\alpha$ of an independent variable. 

Focusing on conformable derivatives, we derived fundamental calculus rules—including scaling, product, quotient, chain, and order-shift properties—with complete mathematical rigor. Furthermore, we constructed the conformable multiorder integral operator as a formal inverse and demonstrated its application in obtaining exact analytical solutions to several classes of non-linear multiorder differential equations. This operational framework provides a solid foundation that can be readily extended to other local and non-local fractional derivatives.

\section{Final Comments and Future Directions}\label{sec:final_comments}

The multiorder operational framework developed in this paper opens several avenues for further theoretical and applied research:

\begin{enumerate}
    \item \textbf{Multivariable Multiorder Calculus:} Extension to multiorder partial derivatives, multiorder gradient, divergence, and curl vector operators on manifolds.
    \item \textbf{Complex and Variable-Order Extensions:} Formulation of multiorder operators with complex-valued orders $\alpha(z), \beta(z)$ or time-varying orders $\alpha(t), \beta(t)$.
    \item \textbf{Differential Equations and Numerical Methods:} Development of analytical methods and stable numerical schemes (e.g., finite difference, finite element methods) for multiorder differential, integral, and integro-differential equations in engineering and physical modeling.
    \item \textbf{Linear Systems Analysis:} Investigation of stability, controllability, and observability in state-space systems governed by multiorder operators.
    \item \textbf{Functional Spaces:} Construction of multiorder Sobolev and Besov spaces, establishing embedding theorems and trace properties.
    \item \textbf{Dynamical and Stochastic Systems:} Analysis of deterministic and stochastic multiorder dynamical systems, including multiorder Brownian motion and anomalous transport processes.
    \item \textbf{Fractal and Complex Dynamics:} Modeling physical phenomena on fractal sets using multiorder differential operators.
    \item \textbf{Interdisciplinary Applications:} Application of multiorder mathematical modeling to solve real-world problems in viscoelasticity, electrochemistry, control theory, and quantitative finance.
    \item \textbf{Optimization and Algorithms:} Design of optimization algorithms and machine learning architectures (e.g., Physics-Informed Neural Networks) incorporating multiorder fractional gradients.
    \item \textbf{Generalized Base Operators:} Construction of multiorder operators using alternative foundational fractional derivatives, such as Riemann-Liouville, Caputo, Hadamard, or generalized Mittag-Leffler memory kernels.
    \item \textbf{Identities and Inequalities:} Derivation of multiorder integral inequalities (e.g., Hermite-Hadamard, Cauchy-Schwarz, Gronwall-type inequalities).
\end{enumerate}

These prospective developments hold significant potential across diverse areas of applied mathematics, including numerical linear algebra, probability theory, geometric algebra, differential geometry, and functional analysis.

\end{document}